\documentclass{amsart}
\usepackage{amssymb,latexsym, amsmath, amscd}
\usepackage{enumerate}
\usepackage{fullpage}
\usepackage{stmaryrd}

\newtheorem{theorem}{Theorem}
\newtheorem{lemma}[theorem]{Lemma}
\newtheorem{proposition}[theorem]{Proposition}

\newtheorem{remark}[theorem]{Remark}
\numberwithin{equation}{section}

\def\Malg{\mathcal{M}}
\def\M{\mathcal{M}}

\def\E{\mathcal{E}}

\def\g{\mathfrak{g}}
\def\m{\mathfrak{m}}
\def\glie{{\mathfrak{g}_{\mathrm{Lie}}}}
\def\mlie{{\mathfrak{m}_{\mathrm{symm}}}}
\def\gann{{\mathfrak{g}^{\mathrm{ann}}}}
\def\mann{{\mathfrak{m}^{\mathrm{anti}}}}

\def\f{\mathbf k}

\def\Leib{\mathrm{Leib}}

\begin{document}

\title[DG Lie algebras and Leibniz algebra cohomology]{Differential graded Lie algebras and Leibniz algebra cohomology}

\author{Jacob Mostovoy}

\address{Departamento de Matem\'aticas, CINVESTAV Av. IPN 2508 Col. San Pedro Zacatenco, Ciudad de M\'exico,  C.P.\ 07360, Mexico}

\begin{abstract}
In this note, we interpret Leibniz algebras as differential graded Lie algebras. Namely, we consider two functors from the category of Leibniz algebras to 
that of differential graded Lie algebras and show that they naturally give rise to the Leibniz cohomology and the Chevalley-Eilenberg cohomology. As an application, we prove a conjecture stated by Pirashvili in arXiv:1904.00121 [math.KT].

\end{abstract}



\maketitle

\section{Introduction}

Perhaps, the simplest non-trivial example of a differential graded (abbreviated as ``{DG}'') Lie algebra is the \emph{cone} on a Lie algebra $\g$.  It consists of two copies of $\g$  placed in degrees 0 and 1 with the differential of degree $-1$ being the identity map. The universal enveloping algebra of the cone on $\g$ is isomorphic, as a differential graded $U(\g)$-module, to the Chevalley-Eilenberg complex of $\g$ which is a free resolution of the base field considered as a trivial $U(\g)$-module. The cone is, clearly, a functor from Lie algebras to DG Lie algebras; one may ask whether other such functors give rise to useful homology and cohomology theories in a similar way, with the corresponding universal enveloping algebras replacing the Chevalley-Eilenberg complex. 

The correct setting for this question may be that of the Leibniz algebras since they are, essentially, truncations of DG Lie algebras. In this note we observe that the complex calculating the Leibniz cohomology of a Leibniz algebra $\g$ is also obtained from the universal enveloping algebra of a certain DG Lie algebra associated functorially with $\g$. Moreover, we exhibit a Chevalley-Eilenberg complex for Leibniz algebras, which reduces to the usual Chevalley-Eilenberg complex when the bracket of the Leibniz algebra is antisymmetric. In the case of trivial coefficients, the cohomology calculated by this Chevalley-Eilenberg complex is nothing more but the cohomology of the maximal Lie quotient $\glie$ of $\g$; however, the complex itself is different from the Chevalley-Eilenberg complex of $\glie$. 

We then use this interpretation of Leibniz homology in order to solve the conjecture of T.~Pirashvili about the vanishing of the homology of a certain graded Lie algebra complex associated with a free Leibniz algebra.

The close connection between Leibniz algebras and DG Lie algebras is implicit in the works of J.-L.~Loday and T.~Pirashvili. In particular, in \cite{LP}, they show that Leibniz algebras can be thought of as Lie algebra objects in the category of truncated chain complexes. There is a non-linear analogue of this relationship; namely, the correspondence between augmented racks and cubical monoids, see \cite{Clau, MRacks}.

\section{Definitions}
\subsection{Differential graded Lie algebras} Recall that a {differential graded Lie algebra} is a Lie algebra in the tensor category of chain complexes. Explicitly, it is a graded vector space $$L =\bigoplus_{i\in \mathbb{Z}}\, L_i$$ over a field $k$ of characteristic zero, equipped with a bilinear bracket $\llbracket-,-\rrbracket:\, L_i\otimes L_j\to L_{i+j}$ satisfying the graded antisymmetry
$$\llbracket x,y\rrbracket  = (-1)^{|x||y|+1}\llbracket y,x\rrbracket ,$$
and the graded Jacobi identity:
$$ (-1)^{|x||z|}\llbracket x,\llbracket y,z\rrbracket \rrbracket +(-1)^{|y||x|}\llbracket y,\llbracket z,x\rrbracket \rrbracket +(-1)^{|z||y|}\llbracket z,\llbracket x,y\rrbracket \rrbracket =0,$$
together with a differential $d:L_i\to L_{i-1}$ which satisfies the graded Leibniz rule:
$$d\llbracket x,y\rrbracket =\llbracket dx,y\rrbracket +(-1)^{|x|}\llbracket x,dy\rrbracket .$$
Here $x$, $y$ and $z$ are arbitrary homogeneous elements in $L$. 

The component $L_0$ is a Lie algebra. A DG Lie algebra $L$ is \emph{non-negatively graded} if $L_i=0$ for $i<0$.

\subsection{Leibniz algebras} A left
Leibniz algebra $\g$ over a field $\f$ of characteristic zero is a vector space with a bilinear bracket $[-,-]$ satisfying the left Leibniz identity 
$$[[x, y], z] = [x,[y,z]] - [y, [x, z]]$$ for all $x,y,z\in \g$. The definition of a right Leibniz algebra is similar, with the left  Leibniz identity replaced with the right Leibniz identity
$$ [x,[y,z]]=[[x,y],z]- [[x,z],y].$$
If $\g$ is a left Leibniz algebra, the right Leibniz algebra $\g^{\mathrm{opp}}$ coincides with $\g$ as a vector space and has the bracket
$$[x,y]_{\g^{\mathrm{opp}}}=[y,x]_{\g}.$$
In the same fashion one defines the opposite of a right Leibniz algebra.
We will mostly speak about left Leibniz algebras and omit the term ``left'' when it cannot lead to confusion.

The \emph{kernel} of $\g$ is the ideal $\gann$ linearly spanned by all elements of the form $[x,x]$ with $x\in\g$. Taking the quotient by the kernel amounts to enforcing antisymmetry in $\g$;  the Lie algebra $$\glie:=\g/\gann$$ is the maximal Lie quotient of $\g$. The Lie algebra $\glie$ acts on $\g$ on the left: if $d:\g\to\glie$ denotes the quotient map, then, for each $x\in\g$ we have $$dx\cdot y = [x,y].$$

\section{The Leibniz functor and its adjoints}
\subsection{Enveloping DG Lie algebras of a Leibniz algebra}
For any DG Lie algebra $L$ its degree one part $L_1$ together with the bracket\footnote{called the \emph{derived} bracket; see \cite{KS}.} 
$$[x,y]:=\llbracket dx,y\rrbracket $$ 
is a Leibniz algebra which we denote by $\Leib(L)$. In what follows, we shall restrict our attention to non-negatively graded DG Lie algebras $L$ such that $L_0 = (\Leib(L))_{\mathrm{Lie}}$.  These DG Lie algebras form a category which we denote by $\overline{\mathbf{DGLie}}$; the functor $$\overline{\mathbf{DGLie}} \xrightarrow{\Leib} \mathbf{Leib}$$ to the category of Leibniz algebras will be called the \emph{Leibniz functor}.  Whenever $\g=\Leib(L)$, we shall say that $L$ is an \emph{enveloping DG Lie algebra} of $\g$.

\begin{lemma}
A DG Lie algebra $L$ is  in $\overline{\mathbf{DGLie}}$ if and only if $L_1\to L_0$ is surjective and the kernel of this map coincides with $d \llbracket L_1,L_1\rrbracket $. 
\end{lemma}
\begin{proof}
Indeed, for $x,y\in L_1$ we have $$d\llbracket x,y\rrbracket  = \llbracket dx, y\rrbracket  - \llbracket x, dy\rrbracket  =  [x, y]_{\Leib(L)} + [y,x]_{\Leib(L)},$$ and, therefore, $d \llbracket L_1,L_1\rrbracket $ coincides with the kernel of the Leibniz algebra  ${\Leib(L)}$. 
\end{proof}

The Leibniz functor has both left and right adjoint functors. The left adjoint functor was defined in \cite{MRacks}. Consider $\g\oplus\glie$ as a chain complex of length two with $\glie$ in degree $0$, $\g$ in degree $1$ and the differential $d:\g\to\glie$ being the quotient map. The free graded Lie algebra on $\g\oplus\glie$ 
is a DG Lie algebra with the differential induced by $d$; 
      let $\E(\g)$ be the quotient of this free DG Lie algebra by the relations
	\begin{equation}\label{relations1}
	\llbracket x, y\rrbracket  = [x,y]_\glie
	\end{equation}
      when $x,y\in \glie$, and 
	\begin{equation}\label{relations2}
	\llbracket x, y\rrbracket  = x\cdot y
	\end{equation}
      when $x\in \glie$ and $y\in \g$. We call $\E(\g)$ the \emph{universal enveloping DG Lie algebra} of $\g$. The following is immediate:
\begin{proposition}
The universal enveloping DG Lie algebra functor is left adjoint to the Leibniz functor.
\end{proposition}

The right adjoint to the Leibniz functor is the \emph{minimal enveloping DG Lie algebra}. It assigns to a Leibniz algebra $\g$ the non-negatively graded 3-term DG Lie algebra $\Malg(\g)$ defined as
$$\ldots\to 0\to 0\to \gann\xrightarrow{i}\g\xrightarrow{d} \glie.$$
Here, $i$ is the inclusion, and the brackets in $\Malg(\g)$ are defined as
$$\llbracket a,b\rrbracket  = [a,b]_{\glie}$$		for $a,b\in\glie$,
$$\llbracket a, x\rrbracket =a\cdot x$$			when $x\in\g$ and $a\in\glie$,
$$i(\llbracket a,x\rrbracket )=a\cdot i(x)$$		for $x\in\gann$ and $a\in\glie$, and
$$\llbracket x,y\rrbracket =[x,y]_\g +[y,x]_\g$$	when $x,y\in\g$. 
It is a straightforward check that $\Malg(\g)$ is, indeed a DG Lie algebra which is a functor of $\g$.
\begin{proposition}
The functor  $\g \mapsto \Malg(\g)$ is right adjoint to the Leibniz functor.
\end{proposition}
\begin{proof}
Consider a DG Lie algebra $L= \ldots\to L_2\xrightarrow{d_2} L_1\to L_0 = L_1/d\llbracket L_1,L_1\rrbracket $. Define the homomorphism $m: L\to \Malg(\Leib(L))$ by setting it to be the identity on $L_1$ and $L_0$ and, in degree 2, to coincide with $d_2: L_2 \to \Leib(L)^{\mathrm{ann}}\subseteq L_1 $. Then, given a Leibniz algebra homomorphism $\phi: \Leib(L)\to\g$ the composition $m\circ \Malg(\phi)$ is the adjoint map. 
\end{proof}
\begin{remark}
In fact, the category of Leibniz algebras is equivalent to the subcategory of the acyclic DG Lie algebras in $\overline{\mathbf{DGLie}}$ which are zero in degrees three and higher, the equivalence being given by the functor $\Malg(\g)$ and the Leibniz functor.
\end{remark}
\begin{remark}
It is not hard to see that while the Leibniz functor could be defined on the category $\mathbf{DGLie}$ of \emph{all} DG Lie algebras, it would fail to have the right adjoint there. In particular, given a homomorphism $\Leib(L)\to \g$ there is no canonical way to define the action of the Lie algebra $L_0$ on $\g$ which is part of the structure of an adjoint functor. This can be resolved by choosing $L_0$ to coincide with $(\Leib(L))_{\mathrm{Lie}}$, that is, considering the category $\overline{\mathbf{DGLie}}$.
\end{remark}
\begin{remark}\label{init}
The terms ``universal enveloping'' and ``minimal enveloping'' for the DG Lie algebras  $\E(\g)$ and $\Malg(\g)$ reflect the fact that, in the category of the enveloping DG Lie algebras of a given Leibniz algebra $\g$, the DG Lie algebra $\E(\g)$ is the initial, and $\Malg(\g)$ is the terminal object.
\end{remark}


\subsection{Representations of Leibniz algebras as DG modules}

A chain complex $M$ is a (DG) module over a DG Lie algebra $L$ if it equipped with a degree zero bracket
$$\llbracket -,-\rrbracket: L\otimes M\to M,$$
which satisifes
$$ \llbracket\llbracket x,y\rrbracket,m\rrbracket =\llbracket x,\llbracket y,m\rrbracket\rrbracket-(-1)^{|y||x|}\llbracket y,\llbracket x,m\rrbracket\rrbracket$$
and
$$d\llbracket x,m\rrbracket=\llbracket dx,m\rrbracket+(-1)^{|x|}\llbracket x,dm\rrbracket.$$
Here $x$ and $y$ are arbitrary homogeneous elements in $L$ and $m$ is a homogeneous element of $M$.

This is the definition of a \emph{left} module; any left module over a DG Lie algebra is also a right module with the bracket
$$\llbracket m,x \rrbracket = (-1)^{|x||m|+1}\llbracket x,m\rrbracket.$$

A representation of a left Leibniz algebra $\g$ is a vector space $\m$ equipped with bilinear brackets $\g\otimes\m\to\m$ and $\m\otimes\g\to\m$
satisfying 
$$[[m, x], y] = [m,[x,y]] - [x, [m, y]],$$
$$[[x, m], y] = [x,[m,y]] - [m, [x, y]],$$
$$[[x, y], m] = [x,[y,m]] - [y, [x, m]],$$
where $x,y\in\g$ and $m\in \m$.

A representation $\m$ of a right Leibniz algebra $\g$ is defined in the similar way; instead of the left, it satisfies the right Leibniz identity where two of the arguments lie in $\g$ and one lies in $\m$. Loday and Pirashvili use the term ``co-representation'' for a representation of the opposite Leibniz algebra. Given a representation $\m$ of $\g$, the representation $\m^{\mathrm{opp}}$ of $\g^{\mathrm{opp}}$ is defined as the same vector space as $\m$ with the brackets
$[x,m]_{\m^{\mathrm{opp}}}=[m,x]_{\m}$
and
$[m,x]_{\m^{\mathrm{opp}}}=[x,m]_{\m}$.

For any $L$-module $M$, the space $M_i$ becomes a representation of  $\Leib(L)$ if we define 
\begin{equation}\label{moda}[x,m]=\llbracket dx, m\rrbracket\end{equation}
and
\begin{equation}\label{modb}[m,x]= - \llbracket x, dm \rrbracket\end{equation}
for $x\in L_1$ and $m\in M_i$. 

\medskip

Given a representation $\m$ of a Leibniz algebra $\g$, denote by $\mann\subseteq\m$ the subspace spanned by all the expressions of the form $[x,m]+[m,x]$ and let $\mlie$ be the quotient $\m/\mann$. Consider the quotient map $d:\m\to\mlie$ as the differential in the DG vector space whose only non-trivial components are $\m$ in degree zero and $\mlie$ in degree $-1$. 

Define $\E(\m)$ as the quotient of the free $\E(\g)$-module generated by the DG vector space $\m\to\mlie$ by the relations (\ref{moda}) and (\ref{modb}).  As a graded (not DG) vector space, $\overline{\E(\m)}:=\E(\m)_{\geq 0}$ coincides with the free module over the free graded Lie algebra generated by $\g$ in degree 1. In particular, for $i\geq 0$ we have $$\E_i(\m) = \m\otimes \g^{\otimes i}.$$
The differential $\E_0(\m)\to\E_{-1}(\m)$ coincides with $\m\to\mlie$.

Consider the DG vector space
$$\M(\m):= \ldots\to 0\to \mann \xrightarrow{i} \m\xrightarrow{d}\mlie\to 0 \to\ldots $$
with $\m$ in degree zero, $i$ the inclusion and $d$  the quotient map.

\begin{proposition}
The DG vector space $\M(\m)$ has a natural structure of an $\Malg(\g)$-module. 
\end{proposition}

\begin{proof}
Let us construct an explicit action of an arbitrary enveloping DG Lie algebra $L$ of $\g$  on $\M(\m)$.

The last of the three conditions satisfied by a representation implies that $$[[x,y]+[y,x],m] =0$$ and therefore, the bracket $\g\otimes\m\to\m$ descends to a Lie algebra action $\glie\otimes\m\to\m$. This action preserves $\mann$ since
$$[x,[y,m]]+[x,[m,y]] = [[x,y],m]+[m,[x,y]]+ [y,[x,m]]+[[x,m],y]$$
and, therefore, descends to an action of $\glie$ on $\mlie$. These three actions together give the action of $L_0$ on $\M(\m)$.

Define the maps $L_1\otimes \M_i(\m)\to \M_{i+1}(\m)$, $i=-1,0$ by
$$x\otimes dm\mapsto - [m,x],$$
$$x\otimes m\mapsto [x,m]+[m,x]$$
and the map $L_2\otimes \M_{-1}(\m)\to \M_{1}(\m)$ by
$$\llbracket x,y \rrbracket\otimes dm\mapsto -[m, [x,y]+[y,x]].$$
Verifying that $\M(\m)$ is indeed an $L$-module is straightforward.
\end{proof}


\begin{remark}
The choice of grading on a DG module over a DG Lie algebra $L$ is arbitrary, since a shift in grading produces another DG module over $L$. Note that a certain ambiguity in our notation: if $\g$ is considered as a representation of itself, the DG module $\M(\g)$ is obtained from the DG Lie algebra $\Malg(\g)$ by shifting the grading down by one. 
\end{remark}


\section{Leibniz algebra homology and cohomology via enveloping DG Lie algebras}

Homology and cohomology for right Leibniz algebras were defined by Loday; 
see \cite{Loday, LP0}\footnote{note that in these standard references, there is a minor 
error in the formula for the differential in cohomology: in the last term, $n$ should be $n+1$, see \cite[page 138]{Loday} and  \cite[page 144]{LP0}}. Here, we are concerned with 
left Leibniz algebras; by the cohomology $HL^*(\g,\m)$ of a left Leibniz algebra $\g$ with coefficients in a representation $\m$ we understand the 
cohomology of the right Leibniz algebra $\g^{\mathrm{opp}}$, with coefficients in $\m^{\mathrm{opp}}$; similarly, the homology $HL_*(\g,\m)$ of $\g$ is the homology of $\g^{\mathrm{opp}}$ with coefficients in $\m$.

\subsection{The Chevalley-Eilenberg complex of a Leibniz algebra}

For any Leibniz algebra $\g$, its minimal enveloping DG Lie algebra $\Malg(\g)$ is acyclic.  As a consequence, the universal enveloping algebra $U(\Malg(\g))$ has no homology in positive dimensions (see Proposition~2.1 of \cite[Appendix B]{Q}). It can be seen that the graded components of $U(\Malg(\g))$ are free as left $U(\glie)$-modules. Therefore, $U(\Malg(\g))$ is a free resolution of the base field $k$ considered as a trivial $U(\glie)$-module; we call it the \emph{Chevalley-Eilenberg complex} of the Leibniz algebra $\g$. When $\g$ is a Lie algebra, $\Malg(\g)$ is the cone on $\g$ and 
$U(\Malg(\g))$ is the usual Chevalley-Eilenberg complex.

\begin{proposition}
Let $\m$ be a module over the Lie algebra $\glie$. The chain complex $\m\otimes U(\Malg(\g))$ calculates the homology, and the cochain complex
$\mathrm{Hom}(U(\Malg(\g)), \m)$ the cohomology of the maximal Lie quotient $\glie$ with coefficients in $\m$. Here, the tensor product and $\mathrm{Hom}$ are taken in the category of $\glie$-modules.
\end{proposition}

\begin{proof} 
It only remains to see that the graded components of $U(\Malg(\g))$ are free as left $U(\glie)$-modules. Indeed, $U(\Malg(\g))$, as a graded 
left $U(\glie)$-module, is isomorphic to $U(\glie)\otimes U(\Malg(\g)_{>0})$, where $\Malg(\g)_{>0}$ is the graded Lie algebra obtained from $\Malg(\g)$ by setting the zero degree component to be trivial. Therefore, each graded component of $U(\Malg(\g))$ is free.
\end{proof}

\subsection{Leibniz homology and cohomology with coefficients in a Lie algebra representation} 

One can generalise the definition of the Chevalley-Eilenberg complex by replacing the minimal enveloping DG Lie algebra $\Malg(\g)$ with an arbitrary enveloping DG Lie algebra functor. The resulting  complex will not necessarily be acyclic, of course; however, it may produce interesting functors of $\g$ instead of the usual homology and cohomology. In particular, one can consider the universal enveloping DG Lie algebra $\E(\g)$.

\begin{proposition}
Let $\m$ be a module over the Lie algebra $\glie$. The chain complex $\m\otimes U(\E(\g))$ calculates the Leibniz homology, and the cochain complex
$\mathrm{Hom}(U(\E(\g)), \m)$ the Leibniz cohomology of $\g$ with coefficients in $\m$. Here, the tensor product and $\mathrm{Hom}$ are taken in the category of $\glie$-modules.
\end{proposition}

\begin{proof}

The universal enveloping algebra $U(\E(\g))$ is isomorphic, as a left $U(\glie)$-module, to 
$$U(\glie) \otimes_k T(\g).$$
Indeed, it is the tensor algebra on $\g\oplus\glie$ modulo the relations $a x - x a = a \cdot x$ for $a\in \glie$ and $x\in\g$. Here, we write the product in $T(\g)$ simply as juxtaposition. Therefore, as left $U(\glie)$-modules,
$$\m\otimes_{\glie} U(\E(\g))=\m\otimes_k T(\g).$$

The differential in $\m\otimes_k T(\g)$ is induced by $d:\g\to\glie$; in particular, for $u\in \m$ and $x_1, \ldots, x_n\in \g$ we have
\begin{multline*}
d(u\otimes  x_1 x_2 \ldots x_n) = u\otimes d(x_1) x_2\ldots x_n 
- u \otimes x_1 d(x_2) \ldots x_n+ \ldots + (-1)^{n+1} u  \otimes x_1 x_2\ldots d(x_n) \\
= \sum_{1 \leq i<j \leq n} (-1)^j u \otimes x_1\ldots \bigl(d(x_j)\cdot x_i\bigr) \ldots \widehat{x}_j\ldots x_n + 
\sum_{1 \leq j \leq n} (-1)^{j+1} ud(x_j) \otimes x_1\ldots \widehat{x}_j\ldots x_n.
\end{multline*}

Taking into the account the fact that $d(x_j)\cdot x_i = [x_j, x_i]_\g$, we see that the differential in $\m\otimes_k T(\g)$ 
satisfies
$$d(m \otimes  x_1 x_2 \ldots x_n) = \sum_{1 \leq i<j \leq n} (-1)^j m \otimes x_1\ldots [x_j, x_i] \ldots \widehat{x}_j\ldots x_n + 
\sum_{1 \leq j \leq n} (-1)^{j+1} [m, x_j] \otimes x_1\ldots \widehat{x}_j\ldots x_n.$$
Changing the sign, we arrive to the standard complex for the Leibniz homology of the opposite Leibniz algebra $\g^{\mathrm{opp}}$.

The space $\mathrm{Hom}_{\glie} (U(\glie) \otimes_k T(\g),\m)$ in degree $n$ coincides with the space of all linear maps $\g^{\otimes n}\to\m$.
The differential in the cohomology complex $\mathrm{Hom} (U(\glie) \otimes_k T(\g),\m)$ satisfies 
\begin{equation}\label{leibcomplex}
(df)(x_1,\ldots, x_n)= \sum_{1 \leq i<j \leq n} (-1)^j f(x_1,\ldots, [x_j, x_i], \ldots \widehat{x}_j\ldots x_n) + 
\sum_{1 \leq j \leq n} (-1)^{j+1} [x_j, f(x_1\ldots \widehat{x}_j\ldots x_n)],
\end{equation}
which is equivalent to the differential for the Leibniz cohomology of $\g^{\mathrm{opp}}$.
\end{proof}

\subsection{Leibniz homology and cohomology with coefficients in a Leibniz algebra representation}

Let $L$ be a non-negatively graded DG Lie algebra and $M$ a DG module over $L$. Write $\overline{M}$ for the chain complex obtained from $M$ by replacing each $M_i$ with $i<0$ by zero. Clearly, $\overline{M}$ may fail to be a DG module over $L$; however, it is a graded module over $L$ considered as a graded (not DG) Lie algebra.

In order to extend the observations of the previous subsections to the more general case of coefficients in a Leibniz algebra representation, we replace the module $\m$ by $\overline{\E(\m)}$ or $\overline{\M(\m)}$ and the category of $\g$-modules by the category of graded, although not differential graded, $\E(\g)$-modules. 

\begin{proposition}
Let $\m$ be a representation of a Leibniz algebra $\g$. The chain complex $\overline{\E(\m)}\otimes U(\E(\g))$ calculates the Leibniz homology, and the cochain complex
$\mathrm{Hom}(U(\E(\g)), \overline{\M(\m)})$ the Leibniz cohomology of $\g$ with coefficients in $\m$. Here, the tensor product and $\mathrm{Hom}$ are taken in the category of graded $\E(\g)$-modules.
\end{proposition}

\begin{proof} As we have already noted before, $\overline{\E(\m)}$ is the free module over the free graded Lie algebra generated by $\g$ in degree one. Therefore, the chain complex $\overline{\E(\m)}\otimes U(\E(\g))$ is isomorphic to $\m\otimes T(\g)$; one verifies directly that it is the standard complex for the Leibniz homology of $\g^{\mathrm{opp}}$ with coefficients in $\m$.

The $n$th term of the cochain complex $\mathrm{Hom}(U(\E(\g)), \overline{\M(\m)})$ consists of $\E(\g)$-module morphisms of degree $-n$, with the differential $$df = f\circ d - (-1)^{\deg f} d\circ f.$$ It can be identified with the space of linear maps 
$\mathrm{Hom}_k(\g^{\otimes n}, \m)$, since $\overline{\M(\m)}$ is generated in degree zero by $\m$. Indeed, let $f_i$ be the restriction of $f\in\mathrm{Hom}(U(\E(\g)), \overline{\M(\m)})$ to $\g^{\otimes i}$. Then,  given $f_n: \g^{\otimes n}\to\m$ the $\E(\g)$-module morphism $f$ is reconstructed by applying the left action of $\E(\g)$. In particular, taking the tensor product with (the identity morphism of) $\g\in\E(\g)$ we obtain the map
$$f_{n+1}: \g^{\otimes (n+1)}\to\overline{\M_1(\m)}$$
$$x_1 \ldots x_{n+1}\mapsto (-1)^{n} \llbracket x_1, f_n(x_2, \ldots, x_{n+1})\rrbracket$$
and, therefore,
\begin{multline*}
(-1)^n d\circ f = (-1)^{2n} (\llbracket dx_1, f_n(x_2, \ldots, x_{n+1})\rrbracket - \llbracket x_1, df_n(x_2, \ldots, x_{n+1})\rrbracket )\\ =  
[x_1, f_n(x_2, \ldots, x_{n+1})]+[f_n(x_2, \ldots, x_{n+1}), x_1],
\end{multline*}
which is the correction term to (\ref{leibcomplex}) in the case of coefficients in a Leibniz algebra representation, see \cite{Loday, LP0}.
\end{proof}

\subsection{Chevalley-Eilenberg homology and cohomology with coefficients in a Leibniz algebra representation}

Consider the chain complex $$\overline{\E(\m)}\otimes_{\E(\g)} U(\M(\g))$$ and the cochain complex
$$\mathrm{Hom}_{\E(\g)}(U(\M(\g)), \overline{\M(\m)}).$$ When $\m$ is a representation of $\glie$, these complexes 
coincide with $\m\otimes_{\glie} U(\M(\g))$ and 
$\mathrm{Hom}_{\glie}(U(\M(\g)), \m)$, respectively. 

Therefore, we may denote the respective homology and cohomology groups by $H_*(\g,\m)$ and $H^*(\g,\m)$.

\begin{proposition}
The groups $H_1(\g,\m)$ and $H^1(\g,\m)$ coincide with the respective Leibniz (co)homology groups $HL_1(\g,\m)$ and $HL^1(\g,\m)$.
The natural map $HL_2(\g,\m)\to H_2(\g,\m)$ is surjective while $H^2(\g,\m)\to HL^2(\g,\m)$ is injective.
\end{proposition}

\begin{proof}
Indeed, the quotient $\E(\g)\to\M(g)$ is the identity in degrees 0 and 1. In degree 2 we have $\E_2(\g)=S^2(\g)$ with $d_2(x,y) = [x,y]+[y,x]$, while $\M_2(\g)=\gann$ and $d_2: \M_2(\g)\to \M_1(\g)=\g$ is injective. Since $S^2(\g)\to\gann$ is surjective, the quotient map 
induces an isomorphism in the homology in degrees 0 and 1 and the same is true for the induced maps 
$$\overline{\E(\m)}\otimes U(\E(\g))\to\overline{\E(\m)}\otimes U(\M(\g))$$
and
$$\mathrm{Hom}(U(\E(\g)), \overline{\M(\m)})\to\mathrm{Hom}(U(\M(\g)), \overline{\M(\m)}).$$
The surjectivity of the homology and the injectivity of the cohomology in degree 2 follows from the surjectivity of the quotient map in degree 2 and the injectivity of $\gann\to\g$.

\end{proof}

\section{The free Lie algebra complex}
Consider the graded tensor algebra $T(V)$ generated by a vector space $V$ in degree 1. Inside $T(V)$, the vector space $V$ generates the free graded Lie algebra $F(V)$ by means of the operation of the (graded) commutator
$$\llbracket x,y\rrbracket = xy - (-1)^{|x||y|} yx.$$
For a Leibniz algebra $\g$, the graded Lie algebra $F(\g)$ is a subcomplex of the standard complex $T(\g)$ that calculates the Leibniz homology of $\g$. Pirashvili in \cite{Pirashvili} conjectures the following:
\begin{proposition}
If the Leibniz algebra $\g$ is free, the homology of $F(\g)$ vanishes in dimensions greater than one. 
\end{proposition}
\begin{proof}
The components of positive degree of the universal enveloping DG Lie algebra $\E(\g)$ form the free graded Lie algebra $\E_{>0}(\g)$ on $\g$. We claim that the differential on $\E_{>0}(\g)$ induced from $\E(\g)$ coincides with the one coming from the Leibniz complex $T(\g)$.  
Indeed, we have natural maps of DG vector spaces
$$\E(\g)\to U(\E(\g))\to k\otimes_{U(\glie)}U(\E(\g))=T(\g)$$
whose composition identifies $\E_{>0}(\g)$ with $F(\g)$.

Now, assume that $\g$ is a free Leibniz algebra on the set $X$ and consider the graded vector space 
$$W= \ldots\to 0\to 0\to W_X \xrightarrow{{\mathrm{id}}}W_X,$$
where $W_X$ is the vector space spanned by $X$, placed in degrees 0 and 1. 
It follows from the definitions that $\E(\g)$ can be identified with the free DG Lie algebra $(F(W),d)$  generated by $W$.

By Proposition~2.1 of \cite[Appendix B]{Q}, since $W$ is acyclic, the tensor algebra $T(W)$ has no homology in positive dimensions and, therefore, the free DG Lie algebra $\E(\g)$ on $W$ is also acyclic;  this implies the statement of the Proposition.

\end{proof}

\end{document}